\documentclass[12pt]{article}

\usepackage{amssymb,amsmath,amsthm}
\usepackage{bm}

\newcommand{\n}{\noindent}
\newcommand{\bb}[1]{\mathbb{#1}}
\newcommand{\cl}[1]{\mathcal{#1}}
\newcommand{\vp}{\varepsilon}

\newcommand{\clk}{\mathcal{K}}

\newcommand{\clm}{\mathcal{M}}

\newcommand{\clr}{\mathcal{R}}
\newcommand{\cls}{\mathcal{S}}

\newcommand{\clx}{\mathcal{X}}

\newtheorem{lem}{Lemma}

\newtheorem*{tm}{Theorem}

\begin{document}

\title{Some Remarks on the Toeplitz Corona problem}

\author{Ronald G.\ Douglas and Jaydeb Sarkar}

\date{}
\maketitle

\begin{abstract}
In a recent paper, Trent and Wick \cite{TTW} establish a strong relation between the corona problem and the Toeplitz corona problem for a family of spaces over the ball and the polydisk. Their work is based on earlier work of Amar \cite{amar}. In this note, several of their lemmas are reinterpreted in the language of Hilbert modules, revealing some interesting facts and raising some questions about quasi-free Hilbert modules. Moreover,  a modest generalization of their result is obtained.
\end{abstract}
\vspace{.5in}

\n \textsc{2000 Mathematical Subject Classification: 46E50, 46E22, 47B32, 50H05.} 

\n \textsc{Key Words and Phrases:} \ Corona problem, Toeplitz corona problem, Hilbert modules, Kernel Hilbert spaces.

\vfill

\n \footnoterule

\n {\footnotesize This research was supported in part by a grant from the National Science Foundation.}
\newpage 

\newpage 

\vspace{.5in}

\setcounter{section}{0}

\section{Introduction}\label{sec0}

\indent 

While isomorphic Banach algebras of continuous complex-valued functions with the supremum norm can be defined on distinct topological spaces, the results of Gelfand (cf. \cite{DB}) showed that for an algebra $A\subseteq C(X)$, there is a canonical choice of domain, the maximal space of the algebra. If the algebra $A$ contains the function 1, then its maximal ideal space, $M_A$, is compact. Determining $M_A$ for a concrete algebra is not always straightforward. New points can appear, even when the original space $X$ is compact, as the disk algebra, defined on the unit circle $T$, demonstrates. If $A$ separates the points of $X$, then one can identify $X$ as a subset of $M_A$ with a point $x_0$ in $X$ corresponding to the maximal ideal of all functions in $A$ vanishing at $x_0$. When $X$ is not compact, new points must be present but there is still the question of whether the closure of $X$ in $M_A$ is all of $M_A$ or does there exist a ``corona'' $M_A\backslash X\ne \emptyset$.

The celebrated theorem of Carleson states that the algebra $H^\infty({\bb D})$ of bounded holomorphic functions on the unit disk ${\bb D}$ has no corona. There is a corona problem for $H^\infty(\Omega)$ for every domain $\Omega$ in ${\bb C}^m$ but a positive solution exists only for the case $m=1$ with $\Omega$ a finitely connected domain in ${\bb C}$.

One can show with little difficulty that the absence of a corona for an algebra $A$ means that for $\{\varphi_i\}^n_{i=1}$ in $A$, the statement that 
\begin{itemize}
\item[(1)] $\sum\limits^n_{i=1} |\varphi_i(x)|^2 \ge \vp^2>0$ for all $x$ in $X$ 
\end{itemize}
is equivalent to 
\begin{itemize}
\item[(2)] the existence of functions $\{\psi_i\}^n_{i=1}$ in $A$ such that 
 $\sum\limits^n_{i=1} \varphi_i(x) \psi_i(x)=1$ for $x$ in $X$ .
\end{itemize}

The original proof of Carleson \cite{carleson} for $H^\infty({\bb D})$ has been simplified over the years but the original ideas remain vital and important. One attempt at an alternate approach, pioneered by Arveson \cite{Ar} and Shubert \cite{S}, and extended by Agler-McCarthy \cite{AM}, Amar \cite{amar}, and finally Trent--Wick \cite{TTW} for the ball and polydisk, involves an analogous question about Toeplitz operators. In particular, for $\{\varphi_i\}^n_{i=1}$ in $H^\infty(\Omega)$ for $\Omega = {\bb B}^m$ or ${\bb D}^m$, one considers the Toeplitz operator $T_\Phi\colon \ H^2(\Omega)^n\to H^2(\Omega)$ defined $T_\Phi \bm{f} = \sum_{i=1}^{n} \varphi_if_i$ for $\bm{f}$ in $H^2(\Omega)$, where $\bm{f} = f_1 \oplus \cdots \oplus f_n$ and $\clx^n = \clx \oplus \cdots \oplus \clx$ for any space $\clx$. One considers the relation between the operator inequality
\begin{itemize}
\item[(3)] $T_\Phi T^*_\Phi\ge \vp^2I$ for some $\vp>0$ 
\end{itemize}
and statement (1). One can readily show  that (3) implies that one can solve (2) where the functions $\{\psi_i\}^n_{n=1}$ are in $H^2(\Omega)$. We will call the existence of such functions, statement (4). The original hope was that one would be able to modify the method or the functions obtained to achieve $\{\psi_i\}^n_{i=1}$ in $H^\infty(\Omega)$. That (1) implies (3)
follows from earlier work of Andersson--Carlsson \cite{AC} for the unit ball and of Varopoulos \cite{varopou}, Li \cite{li}, Lin \cite{lin}, Trent \cite{T} and Treil--Wick \cite{TW} for the polydisk.

In the Trent--Wick paper \cite{TTW} this goal was  at least partially accomplished with the use of (3) to obtain a solution to (4) for the case $m=1$ and for the case $m>1$ if one assumes (3) for a family of weighted Hardy spaces. Their method was based on that of Amar \cite{amar}. 

In this note we provide a modest generalization of the result of Trent--Wick in which weighted Hardy spaces are replaced by cyclic submodules or cyclic invariant subspaces of the Hardy space and reinterpretations are given in the language of Hilbert modules for some of their other results. It is believed that this reformulation clarifies the situation and raises several interesting questions about the corona problem and Hilbert modules. Moreover, it shows various ways the Corona Theorem could be established for the ball and polydisk algebras. However, most of our effort is directed at analyzing the proof in \cite{TTW} and identifying key hypotheses.

\section{Hilbert Modules}\label{sec1}

\indent 

A \emph{Hilbert module} over the algebra $A(\Omega)$, for $\Omega$ a bounded domain in ${\bb C}^m$, is a Hilbert space ${\cl H}$ which is a unital module over $A(\Omega)$ for which there exists $C\ge 1$ so that  $\|\varphi\cdot f\|_{\cl H} \le C\|\varphi\|_{A(\Omega)} \|f\|_{\cl H}$ for $\varphi$ in $A(\Omega)$ and $f$ in ${\cl H}$. Here $A(\Omega)$ is the closure in the supremum norm over $\Omega$ of all functions holomorphic in a neighborhood of the closure of $\Omega$.

We consider Hilbert modules with more structure which  better imitate the classical examples of the Hardy and Bergman spaces.

The Hilbert module ${\cl R}$ over $A(\Omega)$ is said to be \emph{quasi-free of multiplicity one} if it has a canonical identification as a Hilbert space closure of $A(\Omega)$ such that:
\begin{itemize}
\item[(1)] Evaluation at a point $z$ in $\Omega$ has a continuous extension to ${\cl R}$ for which the norm is locally uniformly bounded.
\item[(2)] Multiplication by a $\varphi$ in $A(\Omega)$ extends to a bounded operator $T_\varphi$ in ${\cl L}({\cl R})$.
\item[(3)] For a sequence $\{\varphi_k\}$ in $A(\Omega)$ which is Cauchy in ${\cl R}, \varphi_k(z)\to 0$ for all $z$ in $\Omega$ if and only if $\|\varphi_k\|_{\cl R}\to 0$.
\end{itemize}

We normalize the norm on ${\cl R}$ so that $\|1\|_{\cl R}=1$.

We are interested in establishing a connection between the corona problem for ${\cl M}({\cl R})$ and the Toeplitz corona problem on ${\cl R}$. Here ${\cl M}({\cl R})$ denotes the multiplier algebra for ${\cl R}$; that is, ${\cl M}$ consists of the functions $\psi$ on $\Omega$ for which $\psi{\cl R}\subset {\cl R}$. Since 1 is in ${\cl R}$, we see that ${\cl M}$ is a subspace of ${\cl R}$ and hence consists of holomorphic functions on $\Omega$. Moreover, a standard argument shows that $\psi$ is bounded (cf.\cite{DD}) and hence ${\cl M} \subset H^\infty(\Omega)$. In general, ${\cl M}\ne H^\infty(\Omega)$. 

For $\psi$ in ${\cl M}$ we let $T_{\psi}$ denote the analytic Toeplitz operator in ${\cl L}({\cl R})$ defined by module multiplication by $\psi$. Given functions $\{\varphi_i\}^n_{i=1}$ in ${\cl M}$, the set is said to
\begin{itemize}
 \item[(1)] satisfy the corona condition if $\sum\limits^n_{i=1} |\varphi_i(z)|^2\ge \vp^2$ for some $\vp>0$ and all $z$ in $\Omega$;
\item[(2)] have a corona solution if there exist $\{\psi_i\}^n_{i=1}$ in ${\cl M}$ such that $\sum\limits^n_{i=1} \varphi_i(z) \psi_i(z) = 1$ for $z$ in $\Omega$;
\item[(3)] satisfy the Toeplitz corona condition if $\sum\limits^n_{i=1} T_{\varphi_i} T^*_{\varphi_i} \ge \vp^2 I_{\cl R}$ for some $\vp >0$; and
\item[(4)] satisfy the ${\cl R}$-corona problem if there exist $\{f_i\}^n_{i=1}$ in ${\cl R}$ such that
$\sum\limits^n_{i=1} T_{\varphi_i}f_i = 1$ or $\sum\limits^n_{i=1} \varphi_i(z) f(z_i)=1$ for $z$ in $\Omega$ with $\sum\limits^n_{i=1} \|f_i\|^2 \le \frac1{\vp^2}$.
\end{itemize}

\section{Basic implications}\label{sec2}

\indent

It is easy to show that (2) $\Rightarrow$ (1), (4) $\Rightarrow$ (3) and (2) $\Rightarrow$ (4). As mentioned in the introduction, it has been shown that (1) $\Rightarrow$ (3) in case $\Omega$ is the unit ball ${\bb B}^m$ or the polydisk ${\bb D}^m$ and (1) $\Rightarrow$ (2) for $\Omega={\bb D}$ is  Carleson's Theorem. For a class of reproducing kernel Hilbert spaces with complete Nevanlinna-Pick kernels one knows that (2) and (3) are equivalent \cite{BTV} (cf. \cite{AT} and \cite{EP}). These results are closely related to generalizations of the commutant lifting theorem \cite{NF}. Finally, (3) $\Rightarrow$ (4) results from the range inclusion theorem of the first author as follows (cf. \cite{douglas}).

\begin{lem}\label{lem1}
If $\{\varphi_i\}^n_{i=1}$ in ${\cl M}$ satisfy $\sum\limits^n_{i=1} T_{\varphi_i} T^*_{\varphi_i} \ge \vp^2I_{\cl R}$ for some $\vp >0$, then there exist $\{f_i\}^n_{i=1}$ in ${\cl R}$ such that $\sum\limits^n_{i=1} \varphi_i(z) f_i(z)=1$ for $z$ in $\Omega$ and $\sum\limits^n_{i=1} \|f_i\|^2_{\cl R} \le \frac1{\vp^2}$.
\end{lem}

\begin{proof}
The assumption that $\sum\limits^n_{i=1} T_{\varphi_i}T^*_{\varphi_i} \ge \vp^2I$ implies that the operator $X\colon \clr^n \to \clr$ defined by $X \bm{f} = \sum_{i=1}^{n} T_{\varphi_1}f_i$ satisfies $XX^* = \sum\limits^n_{i=1} T_{\varphi_i}T^*_{\varphi_i}\ge \vp^2I_{\cl R}$ and hence by \cite{douglas} there exists $Y\colon {\cl R}\to {\cl R}^n$ such that $XY = I_{\clr}$ with $\|Y\| \le \frac1\vp$. Therefore, with $Y1 = f_1 \oplus\cdots\oplus f_n$, we have $\sum\limits_{i=1}^n \varphi_i(z) f_i(z)= \sum\limits^n_{i=1} T_{\varphi_i}f_i = XY1 = 1$ and $\sum\limits^n_{i=1} \|f_i\|^2_{\cl R} = \|Y1\|^2 \le \|Y\|^2\|1\|^2_{\cl R} \le \frac1{\vp^2}$. Thus the result is proved.
\end{proof}

To compare our results to those in \cite{TTW}, we need the following observations.

\begin{lem}\label{lem2}
Let ${\cl R}$ be the Hilbert module $L^2_a(\mu)$  over $A(\Omega)$ defined to be  the closure of $A(\Omega)$ in $L^2(\mu)$ for some probability measure $\mu$ on \textrm{clos} $\Omega$. For $f$ in $L^2_a(\mu)$, the Hilbert modules $L^2_a(|f|^2\ d\mu)$ and $[f]$, the cyclic submodule of ${\cl R}$ generated by $f$, are isomorphic such that $1\to f$.
\end{lem}

\begin{proof}
Note that $\|\varphi\cdot 1\|_{L^2(|f|^2\ d\mu)} = \|\varphi f\|_{L^2(\mu)}$ for $\varphi$ in $A(\Omega)$ and the closure of this map sets up the desired isomorphism.
\end{proof}

\begin{lem}\label{lem2a}
If $\{f_i\}^n_{i=1}$ are functions in $L^2_a(\mu)$ and $g(z) = \sum\limits^n_{i=1} |f_i(z)|^2$, then $L^2_a(g\ d\mu)$ is isomorphic to the cyclic submodule $[f_1\oplus\cdots\oplus f_n]$ of $L^2_a(\mu)^n$ with $1\to f_1\oplus\cdots\oplus f_n$.
\end{lem}

\begin{proof}
The same proof as before works.
\end{proof}

In \cite{TTW}, Trent--Wick prove this result and use it to replace the $L^2_a$ spaces used by Amar \cite{amar} by weighted Hardy spaces. However, before proceding we want to explore the meaning of this result from the Hilbert module point of view. 
\begin{lem}\label{lem3}
For ${\cl R} = H^2({\bb B}^m)$ (or $H^2({\bb D}^m)$) the cyclic submodule of ${\cl R}^N$  generated by $\varphi_1\oplus\cdots\oplus \varphi_N$ with $\{\varphi_i\}^N_{i=1}$ in $A({\bb B}^m)$ (or $A({\bb D}^m)$) is isomorphic to a cyclic submodule of $H^2({\bb B}^m)$ (or $H^2({\bb D}^m)$).
\end{lem}

\begin{proof}
 Combining Lemma 3 in \cite{TTW} with the observations made in Lemmas \ref{lem2} and \ref{lem2a} above yields the result.
\end{proof}

There are several remarks and questions that arise at this point. First, does this result hold for arbitrary cyclic submodules in $H^2({\bb B}^m)$ or $H^2({\bb D}^m)$, which would require an extension of Lemma 3 in \cite{TTW} to arbitrary $\bm{f}$ in $H^2({\bb B}^m)^n$ or $H^2({\bb D}^m)^n$? (This equivalence follows from the fact that a converse to Lemma \ref{lem2} is valid.) It is easy to see that the lemma can be extended to an $n$-tuple of the form $f_1 h \oplus \cdots \oplus f_n h$, where the $\{f_i\}_{i=1}^{n}$ are in $A(\Omega)$ and $h$ is in $\clr$. Thus one need only assume that the quantities $\{\frac{f_i}{f_j}\}_{i,j =1}^{n}$ are in $A(\Omega)$ or even only equal a.e to some continuous functions on $\partial \Omega$.

Second, the argument works for cyclic submodules in $H^2({\bb B}^m) \otimes \ell^2$ or $H^2({\bb D}^m)\otimes \ell^2$ so long as the generating vectors are in $A(\Omega)$ since Lemma 3 in \cite{TTW} holds in this case also.

Note that since every cyclic submodule of $H^2({\bb D})\otimes \ell^2$ is isomorphic to $H^2({\bb D})$, the classical Hardy space has the property that all cyclic submodules for the case of infinite multiplicity already occur, up to isomorphism, in the multiplicity one case. Although less trivial to verify, the same is true for the bundle shift Hardy spaces of multiplicity one over a finitely connected domain in ${\bb C}$ \cite{AD}. 

Third, one can ask if there are other Hilbert modules ${\cl R}$ that possess the property that every cyclic submodule of ${\cl R}\otimes {\bb C}^n$ or ${\cl R}\otimes \ell^2$ is isomorphic to a submodule of ${\cl R}$? The Bergman module $L^2_a({\bb D})$ does not have this property since the cyclic submodule of $L^2_a({\bb D})\oplus L^2_a({\bb D})$ generated by $1\oplus z$  is not isomorphic to a submodule of $L^2_a({\bb D})$. If it were, we could write the function $1+|z|^2 = |f(z)|^2$ for some $f$ in $L^2_a({\bb D})$ which a simple calculation using a Fourier expansion in terms of $\{z^n\bar z^m\}$ shows is not possible.

We now abstract some other properties of the Hardy modules over the ball and polydisk.

We say that the Hilbert module ${\cl R}$ over $A(\Omega)$ has the \emph{modulus approximation property (MAP)} if for vectors $\{f_i\}^N_{i=1}$ in $\clm \subseteq {\cl R}$, there is a vector $k$ in ${\cl R}$ such that $\|\theta k\|^2_{\cl R} = \sum\limits^N_{j=1} \|\theta f_j\|^2$ for $\theta$ in ${\cl M}$. The map $\theta k\to \theta f_i \oplus\cdots\oplus \theta f_N$ thus extends to a module isomorphism of $[k]\subset {\cl R}$ and $[f_1\oplus\cdots\oplus f_N]\subset {\cl R}^N$. 

For $z_0$ in $\Omega$, let $I_{z_0}$ denote the maximal ideal in $A(\Omega)$ of all functions that vanish at $z_0$. The quasi-free Hilbert module ${\cl R}$ over $A(\Omega)$ of multiplicity one is said to satisfy the \emph{weak modulus approximation property (WMAP)} if 
\begin{itemize}
\item[(1)] A non-zero vector $k_{z_0}$ in ${\cl R}\ominus I_{z_0}\cdot {\cl R}$ can be written in the form $k_{z_0}\cdot 1$, where $k_{z_0}$ is in ${\cl M}$, and $T_{k_{z_0}}$ has closed range acting on ${\cl R}$. In this case $\clr$ is said to have a \emph{good kernel function}.
\item[(2)] Property (MAP) holds for $f_i = \lambda_i k_{z_i}$, $i=1,\ldots, N$ with $0\le\lambda_i\le 1$ and $\sum\limits^N_{i=1} \lambda^2_i=1$.
\end{itemize}

\section{Main result}\label{sec3}

\indent 

Our main result relating properties (2) and (3) is the following one which generalizes Theorem 1 of \cite{TTW}.

\begin{tm}
Let ${\cl R}$ be a (WMAP) quasi-free Hilbert module over $A(\Omega)$ of multiplicity one and $\{\varphi_1\}^n_{i=1}$ be functions in ${\cl M}$. Then the following are equivalent:
\begin{itemize}
\item[\rm (a)] There exist functions $\{\psi_i\}^n_{i=1}$ in  $H^\infty(\Omega)$ such that $\sum\limits^n_{i=1} \varphi_i(z)\psi_i(z)=1$ and $\sum|\psi_i(z)| \le \frac1{\vp^2}$ for some $\vp >0$ and all $z$ in $\Omega$, and
\item[\rm (b)] there exists $\vp>0$ such that for every cyclic submodule ${\cl S}$ of ${\cl R}$, $\sum\limits^n_{i=1} T^{\cl S}_{\varphi_i} T^{{\cl S}^*}_{\varphi_i} \ge \vp^2I_{\cl S}$, where $T^{\cl S}_\varphi = T_\varphi|_{\cl S}$ for $\varphi$ in ${\cl M}$.
\end{itemize}
\end{tm}

\noindent\textit{Proof.}
We follow the proof in \cite{TTW} making a few changes. Fix a dense set $\{z_i\}^\infty_{i=2}$ of $\Omega$.

First, we define for each positive integer $N$, the set ${\cl C}_N$ to be the convex hull of the functions $\left\{\frac{|k_{z_i}|^2}{\|k_{z_i}\|^2}\right\}^N_{i=2}$ and the function $1$ for $i = 1$ with abuse of notation. Since ${\cl R}$ being (WMAP) implies that it has a good kernel function, ${\cl C}_N$ consist of non-negative continuous functions on $\Omega$. For a function $g$ in the convex hull of the set $\left\{\frac{|k_{z_i}|^2}{\|k_{z_i}\|^2}\right\}^N_{i=1}$, the vector $\lambda_1 \frac{k_{z_1}}{\|k_{z_1}\|^2} \oplus\cdots\oplus \lambda_N \frac{k_{z_N}}{\|k_{z_N}\|^2}$ is in ${\cl R}^N$. By definition there exists $G$ in $\clr$ such that $[G]\cong \left[\lambda_1 \frac{k_{z_1}}{\|k_{z_1}\|} \oplus\cdots\oplus \lambda_N \frac{k_{z_N}}{\|k_{z_N}\|}\right]$ by extending the map $\theta G\to \lambda_1 \frac{\theta k_{z_1}}{\|k_{z_1}\|} \oplus\cdots\oplus \lambda_N \frac{\theta k_{z_N}}{\|k_{z_N}\|}$ for $\theta$ in ${\cl M}$.

Second, let $\{\varphi_1,\ldots,\varphi_n\}$ be in ${\cl M}$ and let $T_\Phi$ denote the column operator defined from ${\cl R}^n$ to ${\cl R}$ by $T_\Phi(f_1 \oplus\cdots\oplus f_n) = \sum\limits^n_{i=1} T_{\varphi_i}f_i$ for $\bm{f} = (f_1\oplus\cdots\oplus f_n)$ in ${\cl R}^n$ and set ${\cl K} = \ker T_\Phi \subset {\cl R}^n$. Fix $\bm{f}$ in ${\cl R}^n$.
Define the function
\[
 {\cl F}_N\colon \ {\cl C}_N \times {\cl K} \to [0,\infty)
\]
by
\[
{\cl F}_N(g,\pmb{h}) = \sum^N_{i=1} \lambda^2_i\left\|\frac{k_{z_i}}{\|k_{z_i}\|}(\pmb{f}-\pmb{h})\right\|^2 \text{ for $\pmb{h}=h_1 \oplus\cdots\oplus h_n$ in } {\cl R}^n,
\]
where $g = \sum\limits_{i=1}^n \lambda^2_i\frac{|k_{z_i}|^2}{\|k_{z_i}\|^2}$ and $\sum\limits_{i=1}^{n} \lambda_i^2 =1$. We are using the fact that the $k_{z_i}$ are in ${\cl M}$ to realize $k_{z_i}(\bm{f}-\bm{h})$ in ${\cl R}^n$.

Except for the fact we are restricting the domain of ${\cl F}_N$ to ${\cl C}_N\times {\cl K}$ instead of ${\cl C}_N \times {\cl R}^n$, this definition agrees with that of \cite{TTW}. Again, as in \cite{TTW}, this function is linear in $g$ for fixed $\bm{h}$ and convex in $\bm{h}$ for fixed $g$. (Here one uses the triangular inequality and the fact that the square function is convex.)

Third, we want to identify ${\cl F}_N(g,\bm{h})$ in terms of the product of Toeplitz operators $(T_\Phi^{{{\cl S}_g}}) (T_\Phi^{{{\cl S}_g}})^*$, where ${\cl S}_g$ is the cyclic submodule of ${\cl R}$ generated by a vector $P$ in $\clr$ as given in Lemma \ref{lem2a} such that the map $P \rightarrow  \left(\lambda_1\ \frac{k_{z_1}}{\|k_{z_1}\|} \oplus\cdots\oplus\right.$ $\left. \lambda_N \frac{k_{z_N}}{\|k_{z_N}\|}\right)$ extends to a module isomorphism with 
$g = \sum\limits^N_{i=1} \lambda_i^2 \frac{|k_{z_i}|^2}{\|k_{z_i}\|^2}$,  $0 \le \lambda_j^2 \le 1$, and $\sum\limits^N_{i=1} \lambda_j^2=1$.

Note for $\pmb{f}$ in ${\cl R}^n$, $\inf\limits_{\pmb{h}\in {\cl K}} {\cl F}_N(g, \pmb{h}) \le \frac1{\vp^2} \|T_\Phi\pmb{f}\|^2$ if $T^{{\cl S}_g}_\Phi(T^{{\cl S}_g}_\Phi)^* \ge \vp^2 I_{{\cl S}_g}$. Thus, if $T^{\cl S}_\Phi(T^{\cl S}_\Phi)^*\ge \vp^2 I_{\cl S}$ for every cyclic submodule of ${\cl R}$, we have $\inf\limits_{\pmb{h}\in {\cl K}} {\cl F}_N(g,\pmb{h}) \le \frac1{\vp^2}\|T_\Phi\pmb{f}\|^2$. Thus from the von~Neumann min-max theorem we obtain $\inf\limits_{\pmb{h}\in {\cl K}} \sup\limits_{g\in {\cl C}_N} {\cl F}_N(g,\pmb{h}) = \sup\limits_{g\in {\cl C}_N} \inf\limits_{\pmb{h}\in {\cl K}} {\cl F}_N(g,\bm{h}) \le \frac1{\vp^2}\|T_\Phi\pmb{f}\|^2$. 

From the inequality $T_\Phi T^*_\Phi\ge \vp^2I_{\cl R}$, we know that there exists $\pmb{f}_0$ in ${\cl R}^n$ such that $\|\bm{f}_0\| \le \frac1{\vp} \|1\| = \frac1\vp$ and $T_\Phi \pmb{f}_0=1$. Moreover, we can find $\pmb{h}_N$ in ${\cl K}$ such that ${\cl F}_N(g,\pmb{h}_N) \le \left(\frac1{\vp^2} + \frac1N\right) \|T_\Phi\pmb{f}_0\|^2 = \frac1{\vp^2} + \frac1N$ for all $g$ in ${\cl C}_N$. In particular, for $g_i = \frac{|k_{z_i}|^2}{\|k_{z_i}\|^2}$, we have $T^{{\cl S}_{g_i}}_\Phi(T^{{\cl S}_{g_i}}_\Phi)^* \ge \vp^2 I_{{\cl S}_{g_i}}$, where $\left\|\frac{k_{z_i}}{\|k_{z_i}\|} (\pmb{f}_0- \pmb{h}_N)\right\|^2 < \frac{1}{{\vp}^2} + \frac1N$.

There is one subtle point here in that 1 may not be in the range of $T^{\cl S}_\Phi$. However, if $P$ is a vector generating the cyclic module ${\cl S}_g$, then $P$ is in $\clm$ and $T_P$ has closed range. To see this recall that the map $$\theta P \rightarrow \lambda_1 \frac{\theta k_{z_1}}{\|k_{z_1}\|} \oplus \cdots \oplus \lambda_N \frac{\theta k_{z_N}}{\|k_{z_N}\|}$$ for $\theta$ in $\clm$ is an isometry. Since the functions $\{\frac{k_{z_i}}{\|k_{z_i}\|}\}_{i=1}^N$ are in $\clm$ by assumption, it follows that the operator $M_P$ is bounded on $\clm \subseteq \clr$ and has closed range on $\clr$ since the operators $M_{\frac{k_{z_i}}{\|k_{z_i}\|}}$ have closed range, again by assumption. Therefore, find a vector $\pmb{f}$ in ${\cl S}_g^n$ so that $T_\Phi\pmb{f} = P$. But if $\pmb{f} = f_1\oplus\cdots\oplus f_n$, then $f_i$ is in $[P]$ and hence has the form $f_i= P \tilde f_i$ for $\tilde f_i$ in $\clr$. Therefore, $T_\Phi T_P\tilde{\pmb{f}} = P$ or $T_\Phi\tilde{\pmb{f}} = 1$ which is what is needed since in the proof $\pmb{f}_0 - \tilde{\pmb{f}}$ is in ${\cl K}$.

To continue the proof we need the following lemma.

\begin{lem}
If $z_0$ is a point in $\Omega$ and $\pmb{h}$ is a vector in ${\cl R}^n$, then $\|\pmb{h}(z_0)\|^2_{{\bb C}^n} \le \left\|\frac{k_{z_0}}{\|k_{z_0}\|}\pmb{h}\right\|^2$.
\end{lem}

\begin{proof}
Suppose $\bm{h} = h_1 \oplus \cdots \oplus h_n$ with $\{h_i\}_{i=1}^n$ in $A(\Omega)$. Then $T^*_{h_i} k_{z_0} = \overline{h_i(z_0)} k_{z_0}$ and hence $$\overline{h_i(z_0)} \|k_{z_0}\|^2 = < T^*_{h_i} k_{z_0}, k_{z_0}> = < k_{z_0}, T_{h_i} k_{z_0}>$$
since $T_{k_{z_0}} h_i = T_{h_i} k_{z_0}$. (We are using the fact the $k_{z_0} h_i = k_{z_0} h_i\cdot1 = h_i k_{z_0}\cdot1 = h_i k_{z_0}$.) Therefore, $$|\overline{h_i(z_0)}| \|k_{z_0}\|^2 = | < k_{z_0}, T_{k_{z_0}} h_i>| \leq \|k_{z_0}\|^2 \|T_{\frac{k_{\bm{z}_0}}{\|k_{\bm{z}_0}\|}} h_i\|,$$ or, 
$$|\overline{h_i(z_0)}| \leq \|T_{\frac{k_{z_0}}{\|k_{z_0}\|}} h_i\|.$$ Finally, 
$$\|\bm{h}(z_0)\|_{\mathbb{C}^n}^2 = \sum_{i=1}^{n} |h_i(z_0)|^2 \leq \| T_{\frac{k_{z_0}}{\|k_{z_0}\|}} \bm{h}\|^2,$$ 
and since both terms of this inequality are continuous in the $\clr$-norm, we can eliminate the assumption that $\bm{h}$ is in $A(\Omega)^n$. \qed

\vspace{0.1in}

Returning to the proof of the theorem, we can apply the lemma to conclude that $\|(\pmb{f}_0 - \pmb{h}_0)(z)\|^2_{{\bb C}^n} \le \left\|\frac{k_{z_i}}{\|k_{z_i}\|}(\pmb{f}_0-\pmb{h}_0)\right\|^2 \le \frac1{\vp^2} + \frac1N$. Therefore, we see that the vector $\pmb{f}_N = \pmb{f}_0 - \pmb{h}_N$ in ${\cl R}^n$ satisfies
\begin{itemize}
 \item[(1)] $T_\Phi(\bm{f}_N-\bm{h}_N) = 1$,
\item[(2)] $\|\bm{f}_N - \bm{h}_N\|_{\clr}^2 \leq \frac{1}{\vp^2}+\frac{1}{N}$ and
\item[(3)] $\|(\pmb{f}_N - \bm{h}_N)(z_i)\|^2_{{\bb C}^n} \le \frac1{\vp^2} + \frac1N$ for $i=1,\ldots, N$.
\end{itemize}
Since the sequence $\{\pmb{f}_N\}^\infty_{N=1}$ in ${\cl R}^n$ is uniformly bounded in norm, there exists a subsequence converging in the weak$^*$-topology to a vector $\pmb{\psi}$ in $\clr^n$. Since weak$^*$-convergence implies pointwise convergence, we see that $\sum\limits^n_{j=1} \varphi_j\psi_j=1$ and $\|\psi_j(z_i)\|^*_{{\bb C}^n} \le \frac1{\vp^2}$ for all $z_i$. Since $\bm{\psi}$ is continuous on $\Omega$ and the set $\{z_i\}$ is dense in $\Omega$, it follows that $\pmb{\psi}$ is in $H^\infty_{{\bb C}^n}(\Omega)$ and $\|\pmb{\psi}\|\le \frac1{\vp^2}$ which concludes the proof.
\end{proof}

Note that we conclude that $\bm{\psi}$ is in $H^{\infty}(\Omega)$ and not in $\clm$ which would be the hoped for result.

One can note that the argument involving the min-max theorem enables one to show that there are vectors $\bm{h}$ in $\clk$ which satisfy $$\| k_{z_i} (\bm{f} - \bm{h}) \|^2 \leq \frac{1}{\vp^2} + \frac{1}{N}.$$  Moreover, this shows that there are vectors $\tilde{\bm{f}}$ so that $T_{\Phi} \tilde{\bm{f}} = 1$, $\| \tilde{\bm{f}}\|^2 \leq \frac{1}{\vp^2} + \frac{1}{N}$, and $\|\tilde{\bm{f}} (z_i)\|^2 \leq \frac{1}{\vp^2} + \frac{1}{N}$ for $i = 1,\ldots, N$. An easy compactness argument completes the proof since the sets of vectors for each $N$ are convex, compact and nested and hence have a point in common.

\section{Concluding comments}

\indent 

With the definitions given, the question arises of which Hilbert modules are (MAP) or which quasi-free ones are (WMAP).  Lemma \ref{lem3} combined with observations in \cite{TTW} show that both $H^2({\bb B}^m)$ and $H^2({\bb D}^m)$ are WMAP. Indeed any $L^2_a$ space for a measure supported on $\partial \mathbb{B}^m$ or the distinguished boundary of $\mathbb{D}^m$ has these properties. One could also ask for which quasi-free Hilbert module $\clr$ the kernel functions $\{k_z\}_{z\in\Omega}$ are in ${\cl M}$ and whether the Toeplitz operators $T_{k_z}$ are invertible operators as they are in the cases of $H^2({\bb B}^m)$ and $H^2({\bb D}^m)$. It seems possible that the kernel functions for all quasi-free Hilbert modules might have these properties when $\Omega$ is strongly pseudo-convex, with smooth boundary. In many concrete cases, the $k_{z_0}$ are actually holomorphic on a neighborhood of the closure of $\Omega$ for $z_0$ in $\Omega$, where the neighborhood, of course, depends on $z_0$. 

Note that the formulation of the criteria in terms of a cyclic submodule $\cls$ of the quasi-free Hilbert modules makes it obvious that the condition $$T_{\Phi}^{\cls} (T_{\Phi}^{\cls})^* \geq \vp^2 I_{\cls}$$ is equivalent to $$T_{\Phi} T_{\Phi}^* \geq \vp^2 I_{\clr}$$ if the generating vector for $\cls$ is a cyclic vector. This is Theorem 2 of \cite{TTW}. Also it is easy to see that the assumption on the Toeplitz operators for all cyclic submodules is equivalent to assuming it for all submodules. That is because $$\|(P_{\cls} \otimes I_{\mathbb{C}^n}) T_{\Phi}^* f\| \geq \|(P_{[\bm{f}]} \otimes I_{\mathbb{C}^n}) T_{\Phi}^* f\|$$ for $f$ in the submodule $\cls$. 

If for the ball or polydisk we knew that the function ``representing'' the modules of a vector-valued function could be taken to be continuous on clos$(\Omega)$ or cyclic, the corona problem would be solved for those cases. No such result is known, however, and it seems likely that such a result is false.

Finally, one would also like to reach the conclusion that the function $\bm{\psi}$ is in the multiplier algebra even if it is smaller than $H^{\infty}(\Omega)$. In the recent paper \cite{CSW} of Costea, Sawyer and Wick this goal is achieved for a family of spaces which includes the Drury-Arveson space. It seems possible that one might be able to modify the line of proof discussed here to involve derivatives of the $\{\varphi_i\}_{i=1}^{n}$ to accomplish this goal in this case, but that would clearly be more difficult.

\vspace{1in}

\n\begin{tabular}{l}
Department of Mathematics\\
\noalign{\vspace{-7pt}}
Texas A\&M University\\
\noalign{\vspace{-7pt}}
College Station, TX\ 77843-3368\\
\noalign{\vspace{-7pt}}
Email address:\ rdouglas@math.tamu.edu\\
\noalign{\vspace{-7pt}}
\phantom{Email address:}\ jsarkar@math.tamu.edu
\end{tabular}

\end{document}